\documentclass{amsart}%
\usepackage{hyperref}
\usepackage{amsmath}%
\usepackage{amsfonts}%
\usepackage{amssymb}%
\usepackage{amsthm}
\usepackage{todonotes}
\usepackage{tikz}
\usepackage{enumitem}
\usepackage{comment}

\renewcommand{\hat}{\widehat}
\renewcommand{\bar}{\overline}

\renewcommand{\S}{\mathfrak{S}}
\newcommand{\bl}{\mathrm{bl}}

\DeclareMathOperator{\supp}{Supp}
\DeclareMathOperator{\udeg}{udeg}
\DeclareMathOperator{\ddeg}{ddeg}

\newtheorem{theorem}{Theorem}
\newtheorem{prop}[theorem]{Proposition}

\newtheorem{lemma}[theorem]{Lemma}

\theoremstyle{definition}
\newtheorem{ex}[theorem]{Example}
\newtheorem{defin}[theorem]{Definition}
\newtheorem{quest}[theorem]{Question}

\theoremstyle{remark}
\newtheorem{remark}[theorem]{Remark}

\title{Self-dual intervals in the Bruhat order}
\author{Christian Gaetz}
\thanks{C.G. was partially supported by an NSF Graduate Research Fellowship under grant No. 1122374} 
\address{Department of Mathematics, Massachusetts Institute of Technology, Cambridge, MA.} 
\email{\href{mailto:gaetz@mit.edu}{gaetz@mit.edu}} 

\author{Yibo Gao} 
\email{\href{mailto:gaoyibo@mit.edu}{gaoyibo@mit.edu}}
\date{\today}

\begin{document}

\begin{abstract}
Bj\"{o}rner-Ekedahl \cite{Bjorner-Ekedahl} prove that general intervals $[e,w]$ in Bruhat order are ``top-heavy", with at least as many elements in the $i$-th corank as the $i$-th rank. Well-known results of Carrell \cite{Carrell} and of Lakshmibai-Sandhya \cite{Lakshmibai-Sandhya} give the equality case: $[e,w]$ is rank-symmetric if and only if the permutation $w$ avoids the patterns $3412$ and $4231$ and these are exactly those $w$ such that the Schubert variety $X_w$ is smooth.  

In this paper we study the finer structure of rank-symmetric intervals $[e,w]$, beyond their rank functions.  In particular, we show that these intervals are still ``top-heavy" if one counts cover relations between different ranks.  The equality case in this setting occurs when $[e,w]$ is self-dual as a poset; we characterize these $w$ by pattern avoidance and in several other ways.     
\end{abstract}

\maketitle

\section{Introduction} \label{sec:intro}


We say a complex projective variety $X$ has a \emph{cellular decomposition} if $X$ is covered by the disjoint open sets $\{C_i\}$, each isomorphic to affine space of some dimension, and such that each boundary $\bar{C_j} \setminus C_j$ is a union of some of the $\{C_i\}$.  Given a variety with such a decomposition, it is natural, following Stanley \cite{Stanley-hard-lefschetz}, to define a partial order $Q^X$ on the $\{C_i\}$ by setting $C_i \leq C_j$ whenever $C_i \subseteq \bar{C_j}$.  

When $X=G/B$, the quotient of a complex semisimple algebraic group by a Borel subgroup, the Bruhat decomposition 
\[
G=\bigsqcup_{w \in W} BwB
\]
induces a cellular decomposition $\{BwB/B \: | \: w \in W\}$ of $X$, where $W$ is the Weyl group of $G$.  In this case the partial order $Q^X$ on $W$ is the well known \emph{Bruhat order}.  For $w \in W$ the closure $X_w=\bar{BwB}/B$ itself has the cellular decomposition $\{BuB/B \: | \: u \in W, \: u \leq w\}$, and so its poset of cells $Q^{X_w}$ is the interval $[e,w]$ in Bruhat order on $W$ below the element $w$.  The varieties $X_w$ are called \emph{Schubert varieties}.  

Much of the structure of the Bruhat order is well-understood combinatorially; see Section \ref{sec:background} for some basic definitions and results.  It is graded with the rank of an element $w$ being the length $\ell(w)$ in the Weyl group, it has minimal element $e$, the identity element of $W$ and maximal element $w_0$, the longest element of $W$.  A great deal of work has been done on the structure of intervals $[e,w]$ in Bruhat order \cite{Bjorner-CW, Bjorner-Wachs-shellable, Tenner}.  Most of this paper will focus on the ``type $A_{n-1}$" case, where the Weyl group $W$ is the symmetric group $\S_n$.

For $w\in W$ and $k=0,1,\ldots,\ell(w)$, let
\[
P_k^w:=\{u\leq w:\ell(u)=k\}.
\]
We call this set the \emph{$k$-th rank of $[e,w]$} and call $P_{\ell(w)-k}^w$ the \emph{$k$-th corank}.  When the element $w$ is well understood, we may simplify our notation and just write $P_k$ instead. We have $P_0^w=\{e\}$ and $P_{\ell(w)}^w=\{w\}$.  Let $\Gamma_w$ (resp. $\Gamma^w$) denote the bipartite graph on $P^w_1 \sqcup P^w_2$ (resp. $P^w_{\ell(w)-1} \sqcup P^w_{\ell(w)-2}$) with edges given by cover relations in Bruhat order (see Figure \ref{fig:34521-graphs} for an example).

\begin{theorem}[Bj\"{o}rner and Ekedahl \cite{Bjorner-Ekedahl}]
\label{thm:Bjorner-top-heavy}
Bruhat intervals are ``top-heavy", that is, for all $0 \leq k \leq \ell(w)/2$,
\[
|P^w_k| \leq |P^w_{\ell(w)-k}|.
\]
\end{theorem}

Given a permutation $\pi \in \S_m$, we say $w \in \S_n$ \emph{avoids} $\pi$ if there are no indices $1 \leq i_1 < \cdots < i_m \leq n$ such that $w(i_1),\ldots,w(i_m)$ are in the same relative order as $\pi(1),\ldots, \pi(m)$.

\begin{theorem}[Carrell; Lakshmibai and Sandhya \cite{Carrell, Lakshmibai-Sandhya}]
\label{thm:smooth-conditions}
The following are equivalent for $w \in \S_n$:
\begin{enumerate}[label=S.\arabic*]
    \item \label{enum:rank-symmetric} the interval $[e,w]$ is rank-symmetric, that is,
    $|P^w_k| = |P^w_{\ell(w)-k}|$ for all $0 \leq k \leq \ell(w)/2$;
    \item \label{enum:smooth-patterns} $w$ avoids $3412$ and $4231$;
    \item \label{enum:smooth-variety} the Schubert variety $X_w$ is smooth.
\end{enumerate}
\end{theorem}

Permutations satisfying the equivalent conditions of Theorem \ref{thm:smooth-conditions} are called \emph{smooth permutations}.   

Theorem \ref{thm:our-top-heavy} shows that, even when $[e,w]$ is rank-symmetric, so that Theorem~\ref{thm:Bjorner-top-heavy} does not give an asymmetry between ranks and coranks, the interval is still ``top heavy" if we also consider cover relations.  For $u \in [e,w]$ we write $\udeg_w(u)$ for the number of $v \in [e,w]$ covering $u$, and $\ddeg_w(u)$ for the number covered by $u$.  A poset is called \emph{self-dual} if it is isomorphic to its dual poset, which has the same elements with the order relation reversed.

\begin{theorem} \label{thm:our-top-heavy}
Let $w \in \S_n$ be a smooth permutation, then
\[
\max_{u \in P^w_1} \udeg_w(u) \leq \max_{u \in P^w_{\ell(w)-1}} \ddeg_w(u),
\]
with equality if and only if $[e,w]$ is self-dual.
\end{theorem}

Stanley wondered \cite{Stanley-hard-lefschetz} if the posets $Q^X$ for $X$ smooth are always self-dual (they are rank-symmetric by the Hard Lefschetz Theorem); although this is the case for many small examples, it is not true for the smooth Schubert variety $X_{34521}$ (see Figure \ref{fig:34521-graphs}).  Theorem \ref{thm:main} below characterizes self-dual intervals in Bruhat order on the symmetric group.

\begin{theorem}\label{thm:main}
The following are equivalent for $w \in \S_n$:
\begin{enumerate}[label=SD.\arabic*]
    \item \label{enum:graphs-isomorphic} the bipartite graphs $\Gamma_w$ and $\Gamma^w$ are isomorphic;
    \item \label{enum:polished-patterns} $w$ avoids the smooth patterns $3412$ and $4231$ from (\ref{enum:smooth-patterns}) as well as $34521$, $45321$, $54123$, and $54312$;
    \item \label{enum:w-is-polished} $w$ is polished (see Definition \ref{def:almost-parabolic});
    \item \label{enum:interval-self-dual} the interval $[e,w]$ in Bruhat order is self-dual.
\end{enumerate}
\end{theorem}

\begin{remark}
In Section \ref{sec:proof-of-3-4} we prove that (\ref{enum:w-is-polished})$\Rightarrow$(\ref{enum:interval-self-dual}) in general finite Coxeter groups, however in Section \ref{sec:counterexamples} we give counterexamples to the other implications in general Coxeter groups.
\end{remark}

The equivalence of (\ref{enum:graphs-isomorphic}) and (\ref{enum:interval-self-dual}) is notable because it implies that self-duality of $[e,w]$ may demonstrated by comparing only two pairs of ranks and coranks.  This is in contrast to the case of rank-symmetry, where Billey and Postnikov \cite{Billey-Postnikov} conjecture that one must check that $|P_i^w|=|P_{\ell(w)-i}^w|$ for around the first $r$ pairs of ranks and coranks, where $r$ is the rank of the Weyl group.  In particular, (\ref{enum:graphs-isomorphic}) gives a new sufficient (but not necessary) condition for the smoothness of $X_w$ which may be checked by comparing only two pairs of ranks and coranks.  See \cite{Richmond-Slofstra-triangle} for discussion of a similar problem in certain infinite Coxeter groups.

\

The remainder of the paper is organized as follows.  In Section \ref{sec:background} we recall background on Bruhat order and give the definition of polished elements. Section \ref{sec:proof} gives the proof of Theorem \ref{thm:main} and Theorem~\ref{thm:our-top-heavy}, with each implication in Theorem~\ref{thm:main} (\ref{enum:graphs-isomorphic})$\Rightarrow$(\ref{enum:polished-patterns}), (\ref{enum:polished-patterns})$\Rightarrow$(\ref{enum:w-is-polished}), and (\ref{enum:w-is-polished})$\Rightarrow$(\ref{enum:interval-self-dual}) occupying a subsection and the proof of Theorem~\ref{thm:our-top-heavy} occupying the last subsection. Finally, Section \ref{sec:counterexamples} shows that Theorem \ref{thm:main} does not extend to other finite Coxeter groups.

\section{Background and definitions} \label{sec:background}

Let $(W,S)$ be a finite Coxeter system; we write $\Delta_S$ for the associated Dynkin diagram (see Bj\"{o}rner and Brenti \cite{Bjorner-Brenti} for basic results and definitions).  For $w\in W$, the \emph{length} $\ell(w)$ is the shortest possible length for an expression $w=s_1\cdots s_{\ell}$ with the $s_i \in S$; such an expression for $w$ of minimal length is called a \emph{reduced expression} or \emph{reduced decomposition}.  The \emph{parabolic subgroup} $W_J$ for $J \subseteq S$ is the subgroup generated by $J$, and $(W_J,J)$ is a Coxeter system.  The unique element of maximum length in $W_J$ is denoted $w_0(J)$.  Each left coset $wW_J$ (resp. right coset $W_J w$) of $W_J$ in $W$ has a unique representative $w^J$ (resp. $^Jw$) of minimal length, and the set of these representatives is the \emph{parabolic quotient} $W^J$ (resp. $^JW$).  Given $J \subseteq S$, each element $w\in W$ may be uniquely written $w=w^Jw_J$ with $w^J \in W^J$ and $w_J \in W_J$ (resp. $w=$ $_Jw ^J w$ with $^Jw$ in $^JW$ and $_Jw$ in  $W_J$) with $J$ and this decomposition satisfies $\ell(w)=\ell(w^J)+\ell(w_J)$; whenever we write an element $w$ as a product of two elements whose lengths sum to $\ell(w)$, we say this product is \emph{length-additive}.  The \emph{support} $\supp(w)$ is the set of $s \in S$ appearing in a given reduced expression for $w$ (it is known that the support does not depend on the reduced expression).

The elements of $T=\{wsw^{-1} \: | \: w \in W, s \in S\}$ are called \emph{reflections}.  For $w \in W$ and $t \in T$, we write $w \leq wt$ whenever $\ell(wt)>\ell(w)$; the \emph{Bruhat order} on $W$ is the transitive closure of this relation.  The Bruhat order is graded, with rank function given by $\ell$, has unique minimal element $e$ and unique maximal element $w_0=w_0(S)$.  If above we instead require that $t \in S$, the resulting partial order is called the \emph{right weak order}, denoted $\leq_R$ (if we require that $t \in S$ and multiply on the left, we obtain the \emph{left weak order} $\leq_L$ on $W$).  We write $[u,w]$ for the interval between $u$ and $w$ in Bruhat order, and $[u,w]_L$ and $[u,w]_R$ for intervals in left and right weak orders, respectively; we also write $[u,w]^J$ for $[u,w] \cap W^J$.

\begin{prop}[See, e.g. \cite{Bjorner-Brenti}]
\label{prop:parabolic-projection-preserves-bruhat}
The map $u \mapsto u^J$ from $W \to W^J$ preserves Bruhat order.
\end{prop}

The \emph{right inversion set} $T_R(w)$ of $w \in W$ is $\{t \in T \: | \: \ell(wt)<\ell(w)\}$; the \emph{right descent set} is $D_R(w)=T_R(w) \cap S$.  We similarly define left inversions and descents by multiplying by $t$ on the left.  It is not hard to check that 
\[
W^J = \{w \in W \: | \: D_R(w) \subseteq S \setminus J\}
\]
and that $D_R(w_0(J))=D_L(w_0(J))=J$.  It is well known that $s \in D_R(w)$ (resp. $s \in D_L(w)$) if and only if $w$ has a reduced expression ending with $s$ (resp. beginning with $s$).

The following characterization of Bruhat order is well known.

\begin{prop} \label{prop:subword-property}
Let $u,w \in W$, then $u \leq w$ if and only if for some (equivalently, for any) reduced expression $w=s_{1} \cdots s_{\ell}$ there is a substring $s_{i_1} \cdots s_{i_k}$ with $i_1 < \cdots < i_k$ which is a reduced expression for $u$.
\end{prop}

\subsection{Billey-Postnikov decompositions}

Let $w \in (W,S)$ and $J \subseteq S$, we say the parabolic decomposition $w=w^J w_J$ is a \emph{Billey-Postnikov decomposition} (or \emph{BP-decomposition}) if 
\[
\supp(w^J) \cap J \subseteq D_L(w_J).
\]

For any $u \in W$ and any $J \subseteq S$, it was shown in \cite{Billey-Fan-Losonczy} that 
\[
[e,u] \cap W_J = [e,m(u,J)]
\]
for some element $m(u,J)\in W$, and we take this as the definition of $m(u,J)$.

\begin{prop}[Richmond and Slofstra \cite{Richmond-Slofstra-fibre-bundles}] \label{prop:bp-implies-parabolic-is-maximal}
If the parabolic decomposition $u=u^Ju_J$ is a BP-decomposition, then $u_J=m(u,J)$.
\end{prop}

\subsection{The symmetric group as a Coxeter group}

Much of the paper will focus on the case of the symmetric group $\S_n$, the Coxeter group of type $A_{n-1}$.  We make the conventions for the symmetric group that the simple generators are $S=\{s_1,...,s_{n-1}\}$ where $s_i$ is the adjacent transposition $(i \: i+1)$.  It is not hard to see that the reflections $T$ are exactly the transpositions $(i j)$, for which we sometimes write $t_{ij}$.

In this case descents and inversions correspond to the familiar notions by the same name which appear in the combinatorics of permutations.  Namely, for $w=w(1) \ldots w(n)$ in one-line notation, $(ij)$, $i<j$ is a right inversion of $w$ if $w(i)>w(j)$ and a right descent if this is true and $j=i+1$.  The length $\ell(w)$ is the number of inversions of $w$, and the longest element $w_0$ is the reversed permutation with one-line notation $n \: n-1 \cdots 2 \: 1$. 

\subsection{Polished elements}

We now define the \emph{polished elements} appearing in the statement of Theorem \ref{thm:main}.

\begin{defin} \label{def:almost-parabolic}
Let $(W,S)$ be a finite Coxeter system, we say that $w\in W$ is \textit{polished} if there exist pairwise disjoint subsets $S_1,...,S_k \subseteq S$ such that each $S_i$ is a connected subset of the Dynkin diagram and coverings $S_i=J_i \cup J_i'$ for $i=1,...,k$ with $J_i \cap J_i'$ totally disconnected so that
\[
w=\prod_{i=1}^k w_0(J_i)w_0(J_i \cap J_i')w_0(J_i')
\]
where the product is taken from left to right as $i=1,2,...,k$ (if the $S_j$ are reordered, we obtain a possibly different polished element).

In light of Theorem \ref{thm:main}, the word ``polished" is meant to indicate that these elements are even nicer than smooth elements.
\end{defin}

\begin{ex}
The following element (shown in Figure~\ref{fig:polished154963287}) with $k=2$, $J_1=\{s_8\}$, $J_1'=\emptyset$, $J_2=\{s_2,s_3,s_4,s_6,s_7\}$, $J_2'=\{s_4,s_5,s_6\}$, and multiplication in the order of 
\begin{align*}
w=&w_0(J_1)w_0(J_2)s_4s_6w_0(J_2')\\
=&123456798\cdot 154328769\cdot 123546789\cdot 123457689\cdot 123765489\\
=&154973268
\end{align*}
is a polished element. Notice that $J_2\cap J_2'=\{s_4,s_6\}$ is totally disconnected.

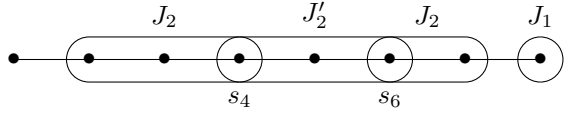
\begin{figure}[h!]
\centering
\begin{tikzpicture}[scale=1.0]
\draw(0,0)--(7,0);
\node at (0,0) {$\bullet$};
\node at (1,0) {$\bullet$};
\node at (2,0) {$\bullet$};
\node at (3,0) {$\bullet$};
\node at (4,0) {$\bullet$};
\node at (5,0) {$\bullet$};
\node at (6,0) {$\bullet$};
\node at (7,0) {$\bullet$};
\draw(1,0.3)--(3,0.3);
\draw(1,-0.3)--(3,-0.3);
\draw (1,0.3) arc (90:270:0.3);
\draw (3,-0.3) arc (-90:90:0.3);
\draw(3,0.3)--(5,0.3);
\draw(3,-0.3)--(5,-0.3);
\draw (3,0.3) arc (90:270:0.3);
\draw (5,-0.3) arc (-90:90:0.3);
\draw(5,0.3)--(6,0.3);
\draw(5,-0.3)--(6,-0.3);
\draw (5,0.3) arc (90:270:0.3);
\draw (6,-0.3) arc (-90:90:0.3);
\draw (7,0.3) arc (90:270:0.3);
\draw (7,-0.3) arc (-90:90:0.3);
\node[above] at (2,0.3) {$J_2$};
\node[above] at (4,0.3) {$J_2'$};
\node[above] at (5.5,0.3) {$J_2$};
\node[below] at (3,-0.3) {$s_4$};
\node[below] at (5,-0.3) {$s_6$};
\node[above] at (7,0.3) {$J_1$};
\end{tikzpicture}
\caption{A polished element 154963287 in $\S_9$.}
\label{fig:polished154963287}
\end{figure}

The permutation $34521 \in \S_5$, whose graphs $\Gamma_{34521}$ and $\Gamma^{34521}$ are shown in Figure \ref{fig:34521-graphs}, is \emph{not} polished.  This can be checked directly or seen to follow from Theorem \ref{thm:main}, since $\Gamma_{34521} \not \cong \Gamma^{34521}$.
\end{ex}

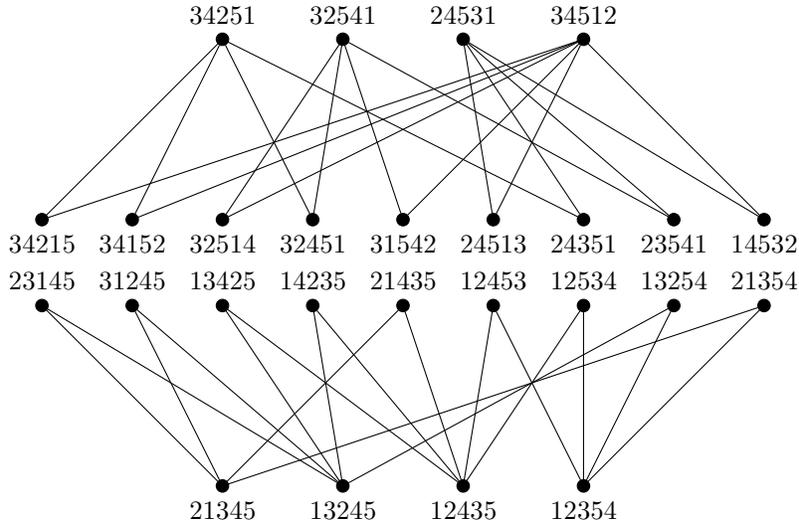
\begin{figure}
    \centering
    \begin{tikzpicture} [scale=0.8]
    \node[draw,shape=circle,fill=black,scale=0.5](b1)[label=above: {$34251$}] at (-3,3) {};
    \node[draw,shape=circle,fill=black,scale=0.5](b2)[label=above: {$32541$}] at (-1,3) {};
    \node[draw,shape=circle,fill=black,scale=0.5](b3)[label=above: {$24531$}] at (1,3) {};
    \node[draw,shape=circle,fill=black,scale=0.5](b4)[label=above: {$34512$}] at (3,3) {};
    
    \node[draw,shape=circle,fill=black,scale=0.5](a1)[label=below: {$34215$}] at (-6,0) {};
    \node[draw,shape=circle,fill=black,scale=0.5](a2)[label=below: {$34152$}] at (-4.5,0) {};
    \node[draw,shape=circle,fill=black,scale=0.5](a3)[label=below: {$32514$}] at (-3,0) {};
    \node[draw,shape=circle,fill=black,scale=0.5](a4)[label=below: {$32451$}] at (-1.5,0) {};
    \node[draw,shape=circle,fill=black,scale=0.5](a5)[label=below: {$31542$}] at (0,0) {};
    \node[draw,shape=circle,fill=black,scale=0.5](a6)[label=below: {$24513$}] at (1.5,0) {};
    \node[draw,shape=circle,fill=black,scale=0.5](a7)[label=below: {$24351$}] at (3,0) {};
    \node[draw,shape=circle,fill=black,scale=0.5](a8)[label=below: {$23541$}] at (4.5,0) {};
    \node[draw,shape=circle,fill=black,scale=0.5](a9)[label=below: {$14532$}] at (6,0) {};
    
    \foreach \y in {(a1),(a2),(a4),(a7)}{
    \draw (b1)--\y;
    }
    
    \foreach \y in {(a3),(a4),(a5),(a8)}{
    \draw (b2)--\y;
    }
    
    \foreach \y in {(a6),(a7),(a8),(a9)}{
    \draw (b3)--\y;
    }
    
    \foreach \y in {(a1),(a2),(a3),(a5),(a6),(a9)}{
    \draw (b4)--\y;
    }
    
    \end{tikzpicture}
    \centering
    \begin{tikzpicture}[scale=0.8]
    \node[draw,shape=circle,fill=black,scale=0.5](a1)[label=below: {$21345$}] at (-3,0) {};
    \node[draw,shape=circle,fill=black,scale=0.5](a2)[label=below: {$13245$}] at (-1,0) {};
    \node[draw,shape=circle,fill=black,scale=0.5](a3)[label=below: {$12435$}] at (1,0) {};
    \node[draw,shape=circle,fill=black,scale=0.5](a4)[label=below: {$12354$}] at (3,0) {};
    
    \node[draw,shape=circle,fill=black,scale=0.5](b1)[label=above: {$23145$}] at (-6,3) {};
    \node[draw,shape=circle,fill=black,scale=0.5](b2)[label=above: {$31245$}] at (-4.5,3) {};
    \node[draw,shape=circle,fill=black,scale=0.5](b3)[label=above: {$13425$}] at (-3,3) {};
    \node[draw,shape=circle,fill=black,scale=0.5](b4)[label=above: {$14235$}] at (-1.5,3) {};
    \node[draw,shape=circle,fill=black,scale=0.5](b5)[label=above: {$21435$}] at (0,3) {};
    \node[draw,shape=circle,fill=black,scale=0.5](b6)[label=above: {$12453$}] at (1.5,3) {};
    \node[draw,shape=circle,fill=black,scale=0.5](b7)[label=above: {$12534$}] at (3,3) {};
    \node[draw,shape=circle,fill=black,scale=0.5](b8)[label=above: {$13254$}] at (4.5,3) {};
    \node[draw,shape=circle,fill=black,scale=0.5](b9)[label=above: {$21354$}] at (6,3) {};
    
    \foreach \y in {(b1),(b2),(b5),(b9)}{
    \draw (a1)--\y;
    }
    
    \foreach \y in {(b1),(b2),(b3),(b4),(b8)}{
    \draw (a2)--\y;
    }
    
    \foreach \y in {(b3),(b4),(b5),(b6),(b7)}{
    \draw (a3)--\y;
    }
    
    \foreach \y in {(b6),(b7),(b8),(b9)}{
    \draw (a4)--\y;
    }
    
    \end{tikzpicture}
    \caption{The bipartite graphs $\Gamma^{34521}$ (top) and $\Gamma_{34521}$ (bottom).  Note that the graphs are not isomorphic.}
    \label{fig:34521-graphs}
\end{figure}

\section{Proof of Theorem~\ref{thm:main}} \label{sec:proof}
It is clear that (\ref{enum:interval-self-dual})$\Rightarrow$(\ref{enum:graphs-isomorphic}), as any antiautomorphism of $[e,w]$ induces an isomorphism $\Gamma_w \cong \Gamma^w$. We are going to show that (\ref{enum:graphs-isomorphic})$\Rightarrow$(\ref{enum:polished-patterns}), (\ref{enum:polished-patterns})$\Rightarrow$(\ref{enum:w-is-polished}) and (\ref{enum:w-is-polished})$\Rightarrow$(\ref{enum:interval-self-dual}) in the following sections.

\subsection{Proof of direction (\ref{enum:graphs-isomorphic})$\Rightarrow$(\ref{enum:polished-patterns})}\label{sub:1-2}

For $w\in \S_n$, let $\bl(w)$ be the largest $b\geq1$ such that $[n]:=\{1,2,\ldots,n\}$ can be partitioned into consecutive intervals $J_1\sqcup J_2\sqcup\cdots\sqcup J_b$ such that $w \cdot J_i=J_i$ for all $i=1,\ldots,b$. We write $w=w^{(1)}\oplus\cdots\oplus w^{(b)}$ where $w^{(i)}\in \S_{|J_i|}$ and say that $w$ has $\bl(w)$ \textit{blocks}. Equivalently, $\bl(w)$ is the cardinality of $S \setminus \supp(w)$, thus we see that $\bl(w)=n-|P_1^w|$.

\begin{defin}\label{def:minimal-inversion}
We say that an inversion $(i,j)$ of $w$ is \textit{minimal} if $i<j$, $w(i)>w(j)$ and there does not exist $k$ such that $i<k<j$ and $w(i)>w(k)>w(j)$.
\end{defin}
In other words, $(i,j)$ is a minimal inversion of $w$ if and only if $w t_{ij}$ is covered by $w$ is in the strong Bruhat order. So the minimal inversions of $w$ are in bijection with $P_{\ell(w)-1}^w$. We generalize this definition to minimal pattern containment.
\begin{defin}\label{def:pattern-contain}
We say that $w\in \S_n$ \textit{contains} pattern $\pi\in \S_k$ at indices $a_1<\cdots<a_k$ if $w(a_i)<w(a_j)$ if and only if $\pi(i)<\pi(j)$ for all $1\leq i<j\leq n$. We say that this occurrence of $\pi$ is \textit{minimal} if there does not exist an occurrence of the pattern $\pi$ at different indices $a_1'<\cdots<a_k'$ such that $a_1'\geq a_1$, $a_k'\leq a_k$, $\min_i w(a_i')\geq\min_i w(a_i)$, $\max_i w(a_i')\leq\max_i w(a_i)$ and at least one of these four inequalities is strict.
\end{defin}
\begin{ex}
The permutation 45321 contains the pattern 3421 at indices 1,2,4,5 but this containment is not minimal since 45321 also contains 3421 at indices 1,2,3,4.
\end{ex}

Notice that if $w\in \S_n$ contains $\pi\in \S_k$, then $w$ must have some minimal occurrence of $\pi$.

\begin{lemma}\label{lem:4231}
For $w\in \S_n$, we always have $|P_{\ell(w)-1}^w|\geq |P_{1}^w|$ and if $w$ contains the pattern $4231$, then $| P_{\ell(w)-1}^w|>|P_{1}^w|$. 
\end{lemma}
\begin{remark}
The inequality $|P_{\ell(w)-1}^w|\geq |P_{1}^w|$ follows directly from Theorem A of \cite{Bjorner-Ekedahl}. We will still give the full proof here as the idea will also be useful later on. 
\end{remark}
\begin{proof}
Use induction on $n$. Let $a=\bl(w)$ and $w=w^{(1)}\oplus\cdots\oplus w^{(a)}$. Then $|P_{\ell(w)-1}^w|=\sum_{i=1}^a|P_{\ell(w^{(i)})-1}^{w^{(i)}}|$ and $|P_1^w|=\sum_{i=1}^a|P^{w^{(i)}}_1|$. As $\bl(4231)=1$, $w$ contains 4231 if and only if one of $w^{(i)}$ contains 4231. Therefore we can assume without loss of generality that $a=1$. Consequently, $P^{w}_1$ consists of all simple transpositions $s_i$ for $i=1,\ldots,n-1$ so $|P^{w}_1|=n-1$.

Let $u\in \S_{n-1}$ be the permutation obtained from $w$ by restricting to the relative ordering of $w(2),\ldots,w(n)$. Let $b=\bl(u)$ and $u=u^{(1)}\oplus\cdots\oplus u^{(b)}$ with $u^{(i)}$ being a permutation on $J_i\subset\{2,\ldots,n\}$. An example is shown in Figure~\ref{fig:4231}.
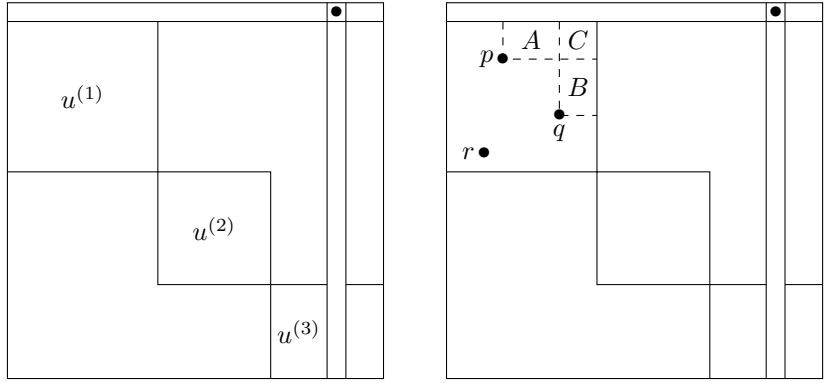
\begin{figure}[h!]
\centering
\begin{tikzpicture}[scale=0.25]
\draw(0,0)--(20,0)--(20,20)--(0,20)--(0,0);
\draw(0,19)--(20,19);
\draw(17,0)--(17,20);
\draw(18,0)--(18,20);
\node at (17.5,19.5) {$\bullet$};
\draw(0,11)--(8,11)--(8,19);
\draw(8,11)--(14,11)--(14,5)--(8,5)--(8,11);
\draw(14,0)--(14,5)--(17,5);
\draw(18,5)--(20,5);
\node at (4,15) {$u^{(1)}$};
\node at (11,8) {$u^{(2)}$};
\node at (15.5,2.5) {$u^{(3)}$};
\end{tikzpicture}
\qquad
\begin{tikzpicture}[scale=0.25]
\draw(0,0)--(20,0)--(20,20)--(0,20)--(0,0);
\draw(0,19)--(20,19);
\draw(17,0)--(17,20);
\draw(18,0)--(18,20);
\node at (17.5,19.5) {$\bullet$};
\draw(0,11)--(8,11)--(8,19);
\draw(8,11)--(14,11)--(14,5)--(8,5)--(8,11);
\draw(14,0)--(14,5)--(17,5);
\draw(18,5)--(20,5);
\node at (2,12) {$\bullet$};
\node[left] at (2,12) {$r$};
\node at (3,17) {$\bullet$};
\node at (3,17)[left] {$p$};
\node at (6,14) {$\bullet$};
\node[below] at (6,14) {$q$};
\draw[dashed](3,19)--(3,17)--(8,17);
\draw[dashed](6,19)--(6,14)--(8,14);
\node at (4.5,18) {$A$};
\node at (7,18) {$C$};
\node at (7,15.5) {$B$};
\end{tikzpicture}
\caption{The decomposition of $w$ with the first entry deleted.  The permutation diagrams in Figures \ref{fig:4231}-\ref{fig:r1after} use matrix coordinates; there is a dot in position $(i,j)$ whenever $w(i)=j$.}
\label{fig:4231}
\end{figure}
Since $\bl(w)=1$, we necessarily have that $w(1)$ is greater than the smallest entry in $J_b$. The minimal inversions of $w$ contain all minimal inversions in $u^{(i)}$'s and minimal inversions of the form $(1,k)$. By the induction hypothesis, the number of minimal inversions in $u^{(i)}$ is at least $|J_i|-1$. And for the minimal inversions in the form of $(1,k)$, we can take $k=w^{-1}(\max J_i-1)$, for $i=1,\ldots,b-1$ (the right most element in each block $u^{(i)}$) and $w^{-1}(w(1)-1)$ (the right most element in the left part of $u^{(b)}$). Together, we obtain $|P^w_{\ell(w)-1}|\geq n-1$ as desired. Moreover, by the induction hypothesis, if any $u^{(i)}$ contains 4231, then the above inequality is strict as well. Thus, we may assume that none of the $u^{(i)}$'s contain 4231.

We now assume that $w$ contains 4231 and all of the 4231's inside $w$ involve the entry $(1,w(1))$. Among all 4231 patterns at indices $1,p,q,r$, choose one where $p$ is minimal and among those, choose one where $w(q)$ is maximal. Since the pattern 231 satisfies $\bl(231)=1$, the entries at $p,q,r$ belong to the same block $J_i$ (see Figure~\ref{fig:4231}). Consider regions $A,B,C$ defined as follows:
\begin{align*}
    A&=\{k\in J_i:k<p,w(p)<w(k)<w(q)\}, \\
    B&=\{k\in J_i:p<k<q,w(q)<w(k)\leq |J_1|+\cdots+|J_i|\}, \\
    C&=\{k\in J_i:k<p,w(q)<w(k)\leq |J_1|+\cdots+|J_i|\}.
\end{align*} 
By minimality of $p$, $A$ must be empty and by maximality of $w(q)$, $B$ must be empty. As $u^{(i)}$ avoids 4231, $C$ must be empty. As a result, $A=B=C=\emptyset$. This means that both $(1,p)$ and $(1,q)$ are minimal inversions of $w$. As $w$ has strictly more than 1 minimal inversions of the form $(1,k)$ for $k\in J_i$, the inequality $|P^w_{\ell(w)-1}|\geq n-1$ is strict, so we are done.
\end{proof}

\begin{lemma}\label{lem:adj-labels}
If $w\in\S_n$ avoids 4231 and has minimal inversions at $(p,q)$ and $(q,r)$, then both $wt_{pq}$ and $wt_{qr}$ cover $wt_{pq}t_{qr}$ and $wt_{qr}t_{pq}$ in the Bruhat interval $[e,w]$.
\end{lemma}
\begin{proof}
We have that $p<q<r$ and $w(p)>w(q)>w(r)$. Since $(p,q)$ and $(q,r)$ are minimal inversions, the sets
\[
\{(a,w(a))\:|\: p<a<q,w(q)<w(a)<w(p)\}
\]
and
\[
\{(a,w(a))\:|\: q<a<r,w(r)<w(a)<w(q)\}
\]
must be empty. Moreover, since $w$ avoids 4231,
\[
\{(a,w(a))\:|\: p<a<q,w(r)<w(a)<w(q)\}
\]
and 
\[
\{(a,w(a))\:|\: q<a<r,w(q)<w(a)<w(p)\}
\]
must be empty as well. As a result, 
\[
\{(a,w(a))\:|\: p<a<r,w(r)<w(a)<w(p)\}=\{(q,w(q))\}.
\]
A useful visualization can be seen in Figure~\ref{fig:adjlabel}.

It is now clear that both $(q,r)$ and $(p,r)$ are minimal inversions of $wt_{pq}$. So $wt_{pq}$ covers $wt_{pq}t_{qr}$ and $wt_{pq}t_{pr}=wt_{qr}t_{pq}$. Similarly, $wt_{qr}$ also covers $wt_{pq}t_{qr}$ and $wt_{qr}t_{pq}$ as desired.
\end{proof}

\begin{lemma}\label{lem:length5patterns}
For $w\in \S_n$ avoiding 4231, if $w$ satisfies (\ref{enum:graphs-isomorphic}) then $w$ avoids 34521, 45321, 54123, 54312 and 3412.
\end{lemma}
\begin{proof}
All four patterns mentioned in this lemma have one block, so we can again without loss of generality assume that $\bl(w)=1$ and therefore that $P_{1}^w=\{s_1,\ldots,s_{n-1}\}$. Assume that $w$ avoids 4231 and it satisfies condition (\ref{enum:graphs-isomorphic}).  Thus there exists some graph isomorphism $\Gamma^w \cong \Gamma_w$ identifying $P_{\ell(w)-1}^w$, which is in bijection with minimal inversions, and $P_1^w$, which is the set of simple transpositions. We will label all minimal inversions by $\{1,2,\ldots,n-1\}$ corresponding to their associated simple transpositions.

The following fact is going to be very useful. Assume $w$ satisfies (\ref{enum:graphs-isomorphic}) and $w$ avoids 4231. Then if $w$ has minimal inversions at $(p,q)$ and $(q,r)$ with labels $i$ and $j$ respectively, then $i$ and $j$ must differ by one (see Figure~\ref{fig:adjlabel}).
\begin{figure}[h!]
\centering
\begin{tikzpicture}[scale=0.25]
\draw(0,0)--(18,0)--(18,12)--(0,12)--(0,0);
\draw(1,0)--(1,12);
\draw(17,0)--(17,12);
\draw(0,11)--(18,11);
\draw(0,1)--(18,1);
\draw(0,5)--(18,5);
\draw(0,6)--(18,6);
\draw(10,0)--(10,12);
\draw(11,0)--(11,12);
\node at (0.5,0.5) {$\bullet$};
\node at (10.5,5.5) {$\bullet$};
\node at (17.5,11.5) {$\bullet$};
\draw(0.5,0.5)--(10.5,5.5)--(17.5,11.5);
\node at (9,2) {$C$};
\node at (12,2) {$D$};
\node at (9,10) {$A$};
\node at (12,10) {$B$};
\node[left] at (5.5,3) {$j$};
\node[right] at (14,8.5) {$i$};
\node[right] at (17.5,11.5) {$p$};
\node[left] at (0.5,0.5) {$r$};
\node[left] at (10.4,5.5) {$q$};
\end{tikzpicture}
\caption{Adjacent labels}
\label{fig:adjlabel}
\end{figure}
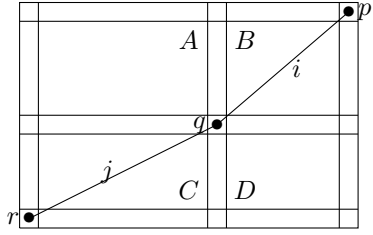

To see this fact, we use Lemma~\ref{lem:adj-labels}. The graph isomorphism $\Gamma_w \cong \Gamma^w$ implies that there exists two elements in $P_{2}^{w}$ that cover both $s_i$ and $s_j$ in the strong Bruhat order. As a result, $|i-j|=1$ since otherwise, there exists only one element $s_is_j=s_js_i\in P_{2}^w$ that covers both $s_i$ and $s_j$.

We first deal with the patterns 34521, 45321, 54123, 54312 of size five. If $w$ contains 45321, take a minimal pattern at indices $a_1<a_2<a_3<a_4<a_5$ and consider the 16 regions indicated in Figure~\ref{fig:45321}. Since $w$ avoids 4231, we know that $A_{11},A_{12},A_{21},A_{22},A_{31},A_{33},A_{34},A_{42},A_{43},A_{44}$ are all empty. If $A_{41}$ is non empty and contains some $(a',w(a'))$, then $w$ contains a pattern 45321 at indices $a_1<a_2<a_3<a_4<a'$, contradicting the minimality of $a_1<a_2<a_3<a_4<a_5$. Similarly, the rest of the regions $A_{13},A_{14},A_{23},A_{24},A_{32}$ are all empty by the minimality. 
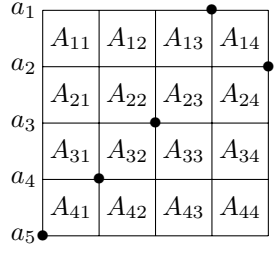
\begin{figure}[h!]
\centering
\begin{tikzpicture}[scale=0.25]
\draw[step=3.0,black] (0,0) grid (12,12);
\node at (0,0) {$\bullet$};
\node at (3,3) {$\bullet$};
\node at (6,6) {$\bullet$};
\node at (9,12) {$\bullet$};
\node at (12,9) {$\bullet$};
\node at (-1,0) {$a_5$};
\node at (-1,3) {$a_4$};
\node at (-1,6) {$a_3$};
\node at (-1,9) {$a_2$};
\node at (-1,12) {$a_1$};
\node at (1.5,1.5) {$A_{41}$};
\node at (1.5,4.5) {$A_{31}$};
\node at (1.5,7.5) {$A_{21}$};
\node at (1.5,10.5) {$A_{11}$};
\node at (4.5,1.5) {$A_{42}$};
\node at (4.5,4.5) {$A_{32}$};
\node at (4.5,7.5) {$A_{22}$};
\node at (4.5,10.5) {$A_{12}$};
\node at (7.5,1.5) {$A_{43}$};
\node at (7.5,4.5) {$A_{33}$};
\node at (7.5,7.5) {$A_{23}$};
\node at (7.5,10.5) {$A_{13}$};
\node at (10.5,1.5) {$A_{44}$};
\node at (10.5,4.5) {$A_{34}$};
\node at (10.5,7.5) {$A_{24}$};
\node at (10.5,10.5) {$A_{14}$};
\end{tikzpicture}
\caption{A minimal 45321.}
\label{fig:45321}
\end{figure}
As a result, we now have minimal inversions at $(a_1,a_3)$, $(a_2,a_3)$, $(a_3,a_4)$ and $(a_4,a_5)$ and let their labels be $i_1,i_2,i_3,i_4$ respectively. By the fact regarding adjacent labels above, we know that $i_3$ is simultaneously adjacent to $i_1$, $i_2$ and $i_4$. This yields a contradiction. We will have the same contradiction if $w$ contains $54312$, the inverse of $45321$.

So we assume further that $w$ avoids 54312 and 45321. If $w$ contains 34521, we similarly take a minimal 34521 at indices $a_1<\cdots<a_5$, and consider the regions shown in Figure~\ref{fig:34521} (left) as before.
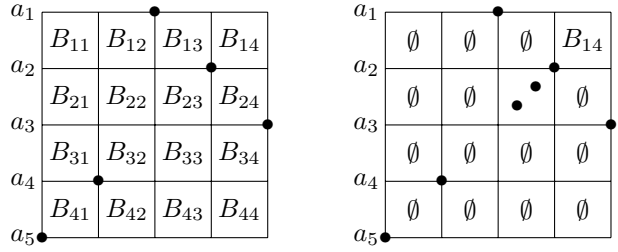
\begin{figure}[h!]
\centering
\begin{tikzpicture}[scale=0.25]
\draw[step=3.0,black] (0,0) grid (12,12);
\node at (0,0) {$\bullet$};
\node at (3,3) {$\bullet$};
\node at (6,12) {$\bullet$};
\node at (9,9) {$\bullet$};
\node at (12,6) {$\bullet$};
\node at (-1,0) {$a_5$};
\node at (-1,3) {$a_4$};
\node at (-1,6) {$a_3$};
\node at (-1,9) {$a_2$};
\node at (-1,12) {$a_1$};
\node at (1.5,1.5) {$B_{41}$};
\node at (1.5,4.5) {$B_{31}$};
\node at (1.5,7.5) {$B_{21}$};
\node at (1.5,10.5) {$B_{11}$};
\node at (4.5,1.5) {$B_{42}$};
\node at (4.5,4.5) {$B_{32}$};
\node at (4.5,7.5) {$B_{22}$};
\node at (4.5,10.5) {$B_{12}$};
\node at (7.5,1.5) {$B_{43}$};
\node at (7.5,4.5) {$B_{33}$};
\node at (7.5,7.5) {$B_{23}$};
\node at (7.5,10.5) {$B_{13}$};
\node at (10.5,1.5) {$B_{44}$};
\node at (10.5,4.5) {$B_{34}$};
\node at (10.5,7.5) {$B_{24}$};
\node at (10.5,10.5) {$B_{14}$};
\end{tikzpicture}
\qquad
\begin{tikzpicture}[scale=0.25]
\draw[step=3.0,black] (0,0) grid (12,12);
\node at (0,0) {$\bullet$};
\node at (3,3) {$\bullet$};
\node at (6,12) {$\bullet$};
\node at (9,9) {$\bullet$};
\node at (12,6) {$\bullet$};
\node at (-1,0) {$a_5$};
\node at (-1,3) {$a_4$};
\node at (-1,6) {$a_3$};
\node at (-1,9) {$a_2$};
\node at (-1,12) {$a_1$};
\node at (1.5,1.5) {$\emptyset$};
\node at (1.5,4.5) {$\emptyset$};
\node at (1.5,7.5) {$\emptyset$};
\node at (1.5,10.5) {$\emptyset$};
\node at (4.5,1.5) {$\emptyset$};
\node at (4.5,4.5) {$\emptyset$};
\node at (4.5,7.5){$\emptyset$};
\node at (4.5,10.5) {$\emptyset$};
\node at (7.5,1.5) {$\emptyset$};
\node at (7.5,4.5) {$\emptyset$};
\node at (7,7) {$\bullet$};
\node at (8,8) {$\bullet$};
\node at (7.5,10.5) {$\emptyset$};
\node at (10.5,1.5) {$\emptyset$};
\node at (10.5,4.5) {$\emptyset$};
\node at (10.5,7.5) {$\emptyset$};
\node at (10.5,10.5) {$B_{14}$};
\end{tikzpicture}
\caption{A minimal 34521.}
\label{fig:34521}
\end{figure}
The cases are slightly more complicated here. Since $w$ avoids 4231, $B_{11},B_{21},B_{31},B_{42},B_{43},B_{44}$ are empty. Since $w$ avoids $45321$, $B_{22},B_{33}$ are empty. Since $a_1<\cdots<a_5$ is minimal, $B_{41},B_{32},B_{12},B_{13},B_{24},B_{34}$ are empty. Thus, among the regions shown in Figure~\ref{fig:34521}, all regions but $B_{23}$ and $B_{14}$ must be empty. Since $w$ avoids 4231, entries in region $B_{23}$ must be decreasing and let them be $(c_1,w(c_1)),\ldots,(c_k,w(c_k))$, $k\geq0$ where $c_1<\cdots<c_k$ and $w(c_1)>\cdots>w(c_k)$, shown in Figure~\ref{fig:3412} (right). By the fact above regarding adjacent labels, we can conclude that the labels of the minimal inversion $(a_4,a_5)$ must be simultaneously adjacent to the labels of $(a_1,a_4)$, $(c_k,a_4)$ and $(a_3,a_4)$ with the convention that $c_0=a_2$. This yields a contradiction. Elements inside region $B_{14}$ will not affect our argument. The case where $w$ contains 54123 is the same as 54123 is the inverse of 34521.

Finally, we can assume that $w$ avoids 4231, 34521, 45321, 54123 and 54312. Suppose that $w$ contains 3412 and let a minimal 3412 be at indices $a_1<a_2<a_3<a_4$. By minimality, all regions except $C_1,C_2,C_3$ must be empty, as shown in Figure~\ref{fig:3412}.
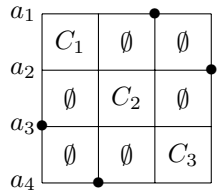
\begin{figure}[h!]
\centering
\begin{tikzpicture}[scale=0.25]
\draw[step=3.0,black] (0,0) grid (9,9);
\node at (0,3) {$\bullet$};
\node at (3,0) {$\bullet$};
\node at (6,9) {$\bullet$};
\node at (9,6) {$\bullet$};
\node at (-1,0) {$a_4$};
\node at (-1,3) {$a_3$};
\node at (-1,6) {$a_2$};
\node at (-1,9) {$a_1$};
\node at (1.5,7.5) {$C_1$};
\node at (4.5,4.5) {$C_2$};
\node at (7.5,1.5) {$C_3$};
\node at (1.5,4.5) {$\emptyset$};
\node at (1.5,1.5) {$\emptyset$};
\node at (4.5,1.5) {$\emptyset$};
\node at (4.5,7.5) {$\emptyset$};
\node at (7.5,4.5) {$\emptyset$};
\node at (7.5,7.5) {$\emptyset$};
\end{tikzpicture}
\caption{A minimal 3412}
\label{fig:3412}
\end{figure}
Since $w$ avoids 4231, elements in $C_2$ must be decreasing. Then as $w$ avoids 45321 (or 54312), $|C_2|\leq2$. We divide into cases depending on the value of $|C_2|$.

If $|C_2|=2$, let it be $(c_1,w(c_1))$ and $(c_2,w(c_2))$ with $c_1<c_2$ and $w(c_1)>w(c_2)$. As $w$ avoids 4231, $C_1$ and $C_3$ must now be empty. The label of the minimal inversion $(c_1,c_2)$ must now be simultaneously adjacent to $(a_1,c_1)$, $(a_2,c_1)$, $(c_2,a_3)$ and $(c_2,a_4)$ and this is clearly impossible. If $|C_1|=1$, let it be $(c_1,w(c_1))$. Similarly $C_1$ and $C_3$ must be empty. Let the labels of the minimal inversions $(a_1,c_1)$, $(a_2,c_1)$, $(c_1,a_3)$ and $(c_1,a_4)$ be $i_1$, $i_2$, $i_3$ and $i_4$ respectively. Then $i_1$ is adjacent to $i_3$, $i_1$ is adjacent to $i_4$, $i_2$ is adjacent to $i_3$ and $i_2$ is adjacent to $i_4$. This is again impossible.

The last remaining case is that $C_2$ is empty so $C_1$ and $C_3$ may not be empty. As $w$ avoids 4231, elements in $C_1$ and $C_3$ are decreasing. Now we use the strategy in the proof of Lemma~\ref{lem:4231} to show that $| P_{\ell(w)-1}^w|>|P_{1}^w|$, contradicting the fact that $w$ was assumed to satisfy (\ref{enum:graphs-isomorphic}). Without of loss generality assume that $\bl(w)=1$ so that $|P_1^w|=n-1$. Let $u$ be obtained from $w$ by removing index $1$ and let $b=\bl(u)$ with blocks $J_1,\ldots,J_b$. Recall that $|P_{\ell(w)-1}|$ is at least the number of minimal inversions inside each block $J_i$ plus the number of minimal inversions involving index $1$ while the number of minimal inversions inside $J_i$ is at least $|J_i|-1$ by induction and the number of minimal inversions involving $1$ and block $J_i$ is at least 1. They sum up to $n-1$. Now if $a_1>1$, since $\bl(3412)=1$, indices $a_1,\ldots,a_4$ together with $C_1$ and $C_3$ must lie in the same block $J_i$ in $u$. We can then use induction to see that the number of minimal inversions inside $J_i$ is strictly larger than $|J_i|-1$ and as a result, $|P_{\ell(w)-1}|>n-1$. The critical case is that $a_1=1$. Let $C_1$ consists of $(c_1,w(c_1)),\ldots,(c_k,w(c_k))$ with $c_1<\cdots<c_k$ and $w(c_1)>\cdots>w(c_k)$, $k\geq0$. Again, indices $a_2,a_3,a_4$ together with $C_1$ and $C_3$ all lie in the same block $J_i$ of $u$. As a result, minimal inversions involving 1 and $J_i$ contain $(1,c_k)$, where $c_0=a_3$ if $k=0$, and $(1,a_4)$, contributing at least 2 to the sum. Therefore, we conclude $| P^w_{\ell(w)-1}|>|P^w_1|$ as well.
\end{proof}

Direction (\ref{enum:graphs-isomorphic})$\Rightarrow$(\ref{enum:polished-patterns}) follows from Lemma~\ref{lem:4231} and Lemma~\ref{lem:length5patterns}.

\subsection{Proof of direction (\ref{enum:polished-patterns})$\Rightarrow$(\ref{enum:w-is-polished})}\label{sub:2-3}
Throughout this section, assume that $w\in\mathfrak{S}_n$ is a permutation that avoids 3412, 4231, 34521, 45321, 54123 and 54312. We are going to use the permutation matrix of $w$, as in Section~\ref{sub:1-2}, to give a decomposition of $w$.

We first divide all such permutations $w$ into different ``types". Consider the region $C=\{(a,w(a))\:|\: 1\leq a\leq w^{-1}(1),1\leq w(a)\leq w(1)\}$ which contains $(1,w(1))$ and $(w^{-1}(1),1)$ and define $t=t(w)=|C|-1$ (see Figure~\ref{fig:selfdualStructureCRL}). If $w(1)=1$, $C$ contains only $(1,1)$ and we say that such $w$ is of type n, where n stands for ``none". We also observe that entries in $C$ are decreasing, meaning that if $(a_1,w(a_1)),(a_2,w(a_2))\in C$ with $a_1<a_2$, then $w(a_1)>w(a_2)$. This is because otherwise, $w$ would contain a pattern 4231 at indices $1,a_1,a_2,w^{-1}$. Assume that $C$ contains $(c_0,w(c_0)),\ldots,(c_t,w(c_t))$ where $1=c_0<\cdots<c_t$ and $w(c_0)>\cdots>w(c_t)=1$.

Then let
\[
R=\{(a,w(a))\:|\: 1<a<w^{-1}(1),w(a)>w(1)\}
\]
and 
\[
L=\{(a,w(a))\:|\: a>w^{-1}(1),1<w(a)<w(1)\}.
\] 
Since $w$ avoids 3412, at least one of $R$ and $L$ must be empty. Otherwise, say $(a_1,w(a_1))\in R$ and $(a_2,w(a_2))\in L$, then automatically $w(1)\neq1$ and $w$ contains a pattern 3412 at indices $1,a_1,w^{-1}(1),a_2$. It is certainly possible that $L=R=\emptyset$, in which case we say that $w$ is of type n as above. If $L\neq\emptyset$, we say that $w$ is of type l, where l stands for either ``left" or ``lower" and if $R\neq\emptyset$, we say that $w$ is of type r, where r stands for ``right". If $w$ is of type l, then $w^{-1}$ is of type r, so these two cases are completely analogous.

\begin{figure}[h!]
\centering
\begin{tikzpicture}[scale=0.4]
\draw(0,0)--(8,0);
\draw(0,0)--(0,-8);
\draw(0,-4)--(8,-4);
\draw(4,0)--(4,-8);
\node at (4,0) {$\bullet$};
\node at (0,-4) {$\bullet$};
\node at (2,-6) {$L$};
\node at (6,-2) {$R$};
\node at (1,-1) {$C$};
\node at (2.4,-2.4) {$\bullet$};
\node at (3.2,-1.2) {$\bullet$};
\node at (1.2,-3.2) {$\bullet$};
\end{tikzpicture}
\qquad
\begin{tikzpicture}[scale=0.4]
\draw(0,0)--(12,0);
\draw(0,0)--(0,-8);
\draw(0,-4)--(12,-4);
\draw(4,0)--(4,-8);
\node at (4,0) {$\bullet$};
\node at (2.8,-0.5) {$\bullet$};
\node at (1.6,-1) {$\bullet$};
\node at (0.8,-2.5) {$\bullet$};
\node at (0,-4) {$\bullet$};
\node at (2,-6) {$\emptyset$};
\node at (0,-4)[left] {$c_t$};
\node at (0,-2.5)[left] {$c_{t-1}$};
\node at (0,-1)[left] {$c_{t-2}$};
\node at (0,0)[left] {$c_0$};
\draw[dashed] (1.6,-1)--(12,-1);
\draw[dashed] (0.8,-2.5)--(12,-2.5);
\node at (11,-0.5) {$\emptyset$};
\node at (11,-1.75) {$R_1$};
\node at (11,-3.25) {$R_0$};
\node at (5,-2.2) {$\bullet$};
\node at (6,-1.9) {$\bullet$};
\node at (7,-1.6) {$\bullet$};
\node at (8,-1.3) {$\bullet$};
\end{tikzpicture}
\caption{Structure of smooth permutations (left) and structure of permutations avoiding 3412, 4231, 34521, 45321, 54123 and 54312 (right).}
\label{fig:selfdualStructureCRL}
\end{figure}
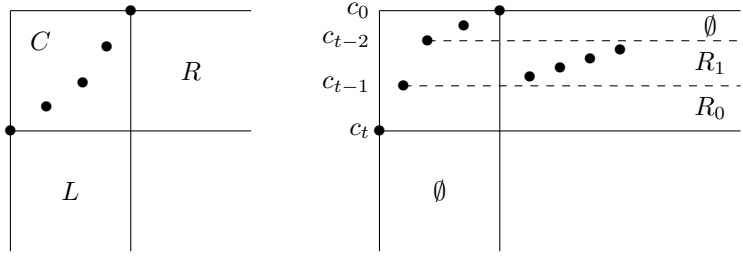

So far we have only used the condition that $w$ is smooth, meaning that $w$ avoids 4231 and 3412. The above analysis has also appeared in previous works including \cite{Gasharov} and \cite{Oh-Postnikov-Yoo}. 

Now assume that $w$ is of type r so that $L=\emptyset$ and $R\neq\emptyset$. We can further divide $R$ as a disjoint union $R_0\sqcup R_1\sqcup R_2$ (see Figure~\ref{fig:selfdualStructureCRL}) where 
\begin{align*}
R_0&=\{(a,w(a))\:|\: c_{t-1}<a<c_t\},\\
R_1&=\{(a,w(a))\:|\: c_{t-2}<a<c_{t-1}\}, \text{ and}\\ R_2&=\{(a,w(a))\:|\: 1<a<c_{t-2}\}.
\end{align*}
As $w$ is of type r, $t\geq1$. If $t=1$, $R_1=R_2=\emptyset$ and if $t=2$, $R_2=\emptyset$ automatically. Regardless, we see that in fact, if $R_2\neq\emptyset$ and contains $(a,w(a))$, then $w$ would contain a pattern 45321 at indices $1,a,c_{t-2},c_{t-1},c_t$. Thus, $R_2=\emptyset$. Moreover, we see that entries in $R_1$ must be decreasing: otherwise if $(a,w(a)),(a',(w(a'))\in R_1$ with $a<a'$ and $w(a)<w(a')$, then $w$ would contain a pattern 34521 at indices $1,a,a',c_{t-1},c_t$, a contradiction. If $R_1\neq\emptyset$, we further say that $w$ is of type r$_1$ and if $R_1=\emptyset$, then $R_0\neq\emptyset$ and we say that $w$ is of type r$_0$. Similarly we can define type l$_1$ and type l$_0$. Equivalently, we can also say that $w$ is of type l$_i$ if $w^{-1}$ is of type r$_i$, $i\in\{0,1\}$.

The following lemma allows us to inductively decompose $w$. As a piece of notation, if $w\in\mathfrak{S}_n$ satisfies $w(1)=1,\ldots,w(m)=m$ for some $m$, then $w$ lies in the parabolic subgroup of $\mathfrak{S}_n$ generated by $J=\{s_{m+1},\ldots,s_{n-1}\}$. In this case, we will naturally consider $w\in(\mathfrak{S}_n)_J$ as a permutation in $\mathfrak{S}_{n-m}$.
\begin{lemma}\label{lem:patternToDecomp}
Let $w\in\mathfrak{S}_n$ be a permutation that avoids the six patterns in (\ref{enum:polished-patterns}). Let $J=\{s_1,\ldots,s_t\}\subset S=\{s_1,\ldots,s_{n-1}\}$ be a connected subset of of the Dynkin diagram of $\mathfrak{S}_n$, where $t=t(w)$ as above.
\begin{itemize}
\item If $w$ is of type n, $w\cdot w_0(J)=w_0(J)\cdot w\in (\mathfrak{S}_n)_{(S\setminus J)\setminus\{s_{t+1}\}}$ is a permutation of size $n-t-1$ that avoids the six patterns in (\ref{enum:polished-patterns}).
\item If $w$ is of type r$_0$, $w_0(J)\cdot w\in(\mathfrak{S}_n)_{S\setminus J}$ is a permutation of size $n-t$ that avoids the six patterns in (\ref{enum:polished-patterns}).
\item If $w$ is of type r$_1$, $w'=s_t\cdot w_0(J)\cdot w\in (\mathfrak{S}_n)_{(S\setminus J)\cup\{s_t\}}$ is a permutation of size $n-t+1$ that avoids the six patterns in (\ref{enum:polished-patterns}). Considered as a permutation in $\S_{n-t+1}$, $t(w')=|R_1|+1$ and $w'$ is not of type r$_1$. Moreover, if $|R_1|=1$, $w'$ is not of type l$_1$ either.
\end{itemize}
\end{lemma}
\begin{proof}
First notice the simple fact that if $u\in\mathfrak{S}_n$ contains one of the patterns in (\ref{enum:polished-patterns}) and $\{u(1),\ldots,u(m)\}=\{1,\ldots,m\}$, then such a pattern appears either within the first $m$ indices or within the last $n-m$ indices.

If $w$ is of type n, then $w(1)=t+1,w(2)=t,\ldots,w(t+1)=1$. After multiplying by $w_0(J)$ on either side, we obtain $w'=w_0(J)w=ww_0(J)$ satisfying $w'(i)=i$ for $i\leq t+1$ and $w'(i)=w(i)$ for $i>t+1$. Clearly $w'$ avoids the patterns of interest, as $w$ avoids them.

If $w$ is of type r$_0$, then $w(1)=t+1,w(2)=t,\ldots,w(t)=2$ and $w(c_t)=1$ where $c_t>t+1$. Let $w'=w_0(J)\cdot w$. We see that $w'(1)=1,\ldots,w'(t)=t$, $w'(c_t)=t+1$ and $w'(i)=w(i)$ if $i\notin\{c_0,\ldots,c_t\}$. So we do have $w'\in(\mathfrak{S}_n)_{S\setminus J}$. By our argument above, if $w'$ contains a pattern $\pi$ mentioned in (\ref{enum:polished-patterns}), then none of the indices $1,\ldots,t$ can be involved, and since $w$ avoids $\pi$, the index $c_t$ must be involved. Say $w'$ contains pattern $\pi$ at indices $a_1<\cdots<a_k$ with $a_i=c_t$. As $a_1>t$, the relative ordering of the entries does not change after we multiply $w$ by $w_0(J)$ on the left to obtain $w'$, so $w$ must also contain pattern $\pi$ at the same indices. This yields a contradiction so $w'$ must avoid all six patterns of interest.

The critical case is that $w$ is of type r$_1$. Let $w'=s_t\cdot w_0(J)\cdot w$ (see Figure~\ref{fig:r1after}). We observe that $w'(i)=i$ for $i\leq t-1$, $w'(c_{t-1})=w(1)$, $w'(c_t)=w(2)$ while $w'$ and $w$ agree on other indices. Thus, $w'$ lies in the parabolic subgroup of $\mathfrak{S}_n$ generated by $s_t,\ldots,s_{n-1}$. We next argue that $w'$ avoids the six patterns of interest. Assume for the sake of contradiction that $w'$ contains one of the patterns in (\ref{enum:polished-patterns}) at indices $a_1<\cdots<a_k$. First, $a_1>t-1$ by the argument above. But when restricted to the last $n-t+1$ indices, $w$ and $w'$ agree by construction, so $w$ must also contain one of the patterns at the same set of indices. This yields a contradiction.

Let $R_1=\{(t,w(t)),\ldots,(t+m-1,w(t+m-1))\}$ where $|R_1|=m$ with $w(t)>\cdots>w(t+m-1)$. Then $c_{t-1}=t+m$. 
Let $w''\in\mathfrak{S}_{n-t+1}$ be the permutation of $w'$ restricted to the last $n-t+1$ indices. In other words, $w''(i)=w'(i+t-1)$. Consider the possible types for $w''$. It is more convenient to stay with the figure of $w'$. If $w''$ were of type r$_1$, then the set 
\[
\{(a,w'(a))\:|\: t<a<t+m,w'(a)>w'(t)\}
\]
cannot be empty, contradicting the fact that entries in $R_1$ are decreasing. Moreover, if $m=|R_1|=1$, $w''$ cannot be of type l$_1$ because otherwise
\[
\{(a,w'(a))\:|\: a>c_t,w'(c_t)<w'(a)<w'(c_{t-1})\}
\]
cannot be empty, contradicting $w$ being type r. It is also evident that $t(w'')=m+1$, as there are $m+2$ entries weakly inside the rectangle bounded by $(t,w'(t))$ and $(c_t,w'(c_t))$.
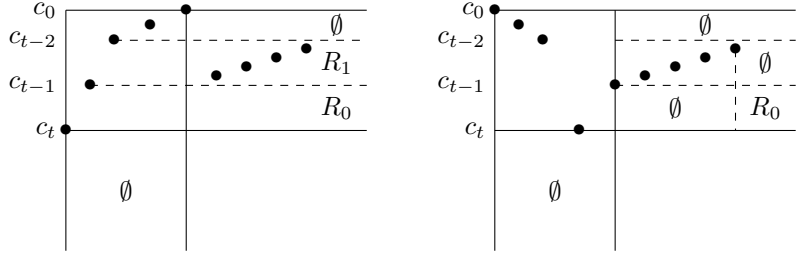
\begin{figure}[h!]
\centering
\begin{tikzpicture}[scale=0.4]
\draw(0,0)--(10,0);
\draw(0,0)--(0,-8);
\draw(0,-4)--(10,-4);
\draw(4,0)--(4,-8);
\node at (4,0) {$\bullet$};
\node at (2.8,-0.5) {$\bullet$};
\node at (1.6,-1) {$\bullet$};
\node at (0.8,-2.5) {$\bullet$};
\node at (0,-4) {$\bullet$};
\node at (2,-6) {$\emptyset$};
\node at (0,-4)[left] {$c_t$};
\node at (0,-2.5)[left] {$c_{t-1}$};
\node at (0,-1)[left] {$c_{t-2}$};
\node at (0,0)[left] {$c_0$};
\draw[dashed] (1.6,-1)--(10,-1);
\draw[dashed] (0.8,-2.5)--(10,-2.5);
\node at (9,-0.5) {$\emptyset$};
\node at (9,-1.75) {$R_1$};
\node at (9,-3.25) {$R_0$};
\node at (5,-2.2) {$\bullet$};
\node at (6,-1.9) {$\bullet$};
\node at (7,-1.6) {$\bullet$};
\node at (8,-1.3) {$\bullet$};
\end{tikzpicture}
\qquad
\begin{tikzpicture}[scale=0.4]
\draw(0,0)--(10,0);
\draw(0,0)--(0,-8);
\draw(0,-4)--(10,-4);
\draw(4,0)--(4,-8);
\node at (0,0) {$\bullet$};
\node at (0.8,-0.5) {$\bullet$};
\node at (1.6,-1) {$\bullet$};
\node at (4,-2.5) {$\bullet$};
\node at (2.8,-4) {$\bullet$};
\node at (2,-6) {$\emptyset$};
\node at (0,-4)[left] {$c_t$};
\node at (0,-2.5)[left] {$c_{t-1}$};
\node at (0,-1)[left] {$c_{t-2}$};
\node at (0,0)[left] {$c_0$};
\draw[dashed] (4,-1)--(10,-1);
\draw[dashed] (4,-2.5)--(10,-2.5);
\node at (5,-2.2) {$\bullet$};
\node at (6,-1.9) {$\bullet$};
\node at (7,-1.6) {$\bullet$};
\node at (8,-1.3) {$\bullet$};
\node at (7,-0.5) {$\emptyset$};
\draw[dashed](8,-1.3)--(8,-4);
\node at (6,-3.25) {$\emptyset$};
\node at (9,-3.25) {$R_0$};
\node at (9,-1.75) {$\emptyset$};
\end{tikzpicture}
\caption{A permutation $w$ of type r$_1$ (left) and the modified permutation $w'=s_t\cdot w_0(J)\cdot w$ (right).}
\label{fig:r1after}
\end{figure}
\end{proof}

We are now ready to prove the implication (\ref{enum:polished-patterns})$\Rightarrow$(\ref{enum:w-is-polished}) by a repeated application of Lemma~\ref{lem:patternToDecomp}. 
\begin{proof}[Proof of implication (\ref{enum:polished-patterns})$\Rightarrow(\ref{enum:w-is-polished})$]
Given $w$ avoiding the six patterns of interest, with $t=t(w)$ and $J=\{s_1,\ldots,s_t\}$, we can obtain $w'\in (\mathfrak{S}_n)_{S'}$ depending on the type of $w$ listed in Table~\ref{tbl:patternDecompose}, by Lemma~\ref{lem:patternToDecomp}.

\begin{table}[h!]
\centering
\begin{tabular}{|c|c|c|}
\hline
type of $w$ & $w'$ & $S'$\\\hline
n & $w_0(J)w=ww_0(J)$ & $\{s_{t+2},\ldots,s_{n-1}\}$\\\hline
r$_0$ & $w_0(J)w$ & $\{s_{t+1},\ldots,s_{n-1}\}$\\\hline
r$_1$ & $s_tw_0(J)w$ & $\{s_{t},\ldots,s_{n-1}\}$\\\hline
l$_0$ & $ww_0(J)$ & $\{s_{t+1},\ldots,s_{n-1}\}$\\\hline
l$_1$ & $ww_0(J)s_t$ & $\{s_{t},\ldots,s_{n-1}\}$\\\hline
\end{tabular}
\caption{A summary of decomposing $w$ after one step}
\label{tbl:patternDecompose}
\end{table}

Continuing with this operation for $w'$ and so on down to the identity, we record each nonempty $J$ as $K^{(1)},K^{(2)},\ldots,K^{(m)}\subset\{s_1,\ldots,s_{n-1}\}$ along the way and assume that $w^{(i)}$ is obtained from $w^{(i-1)}$ as $w'$ is obtained from $w$ above, where we start with $w^{(0)}=w$ and end with $w^{(m)}=\mathrm{id}$. Notice that $J$ is empty if and only if $w(1)=1$, which is equivalent to saying that $w$ is of type n and $t(w)=0$. When $w(1)=1$, we will just consider $w$ as living in the parabolic subgroup generated by $\{s_2,\ldots,s_{n-1}\}$. Assume that $K^{(i)}=\{s_{a_i},\ldots,s_{b_i}\}$, for $a_i\leq b_i$. We label each $K^{(i)}$ by the type of $w^{(i-1)}$. Note that $K^{(m)}$ is of type n.

By Lemma~\ref{lem:patternToDecomp}, if $K^{(i)}$ is of type n, then $b_i<a_{i+1}-1$ which is also saying that any two simple transpositions in $K^{(i)}$ and $K^{(i+1)}$ commute; if $K^{(i)}$ is of type r$_0$ or l$_0$, then $b_i=a_{i+1}-1$ and if $K^{(i)}$ if type r$_1$ or l$_1$, then $b_i=a_{i+1}$ so $K^{(i)}$ and $K^{(i+1)}$ intersects at exactly one position. Moreover, if $K^{(i)}$ is of type r$_1$, then $b_i-a_i\geq1$ and if further $K^{(i+1)}$ is of type l$_1$, then we necessarily have $b_{i+1}-a_{i+1}\geq2$ by Lemma~\ref{lem:patternToDecomp} so that any simple transposition in $K^{(i)}$ and any simple transposition in $K^{(i+2)}$ commute. 

Let $S_1,\ldots,S_k$ be connected components of the Dynkin diagram of $\mathfrak{S}_n$ formed by $K_1,\ldots,K_m$ in this order. We are now going to show that each $S_i$ can be covered by $J_i\cup J_i'$ such that $J_i\cap J_i'$ is totally disconnected and $w$ can be written as the product shown in Definition~\ref{def:almost-parabolic}. This is done by induction on $k$. The base case $k=0$ and $w=\mathrm{id}$ is trivial. Let $S_1=K_1\cup\cdots\cup K_f$. Then $K_1,\ldots,K_{f-1}$ are of types l$_1$ and r$_1$ and are alternating between these two. Without loss of generality, let us assume that $K_1$ is of type r$_1$, since we can invert everything to go from type l$_1$ to type r$_1$. There are the following cases that are almost identical to each other. We will explain the first case in details.

\noindent\textbf{Case 1:} $f=2g-1$ is odd and $K_f$ is of type r$_0$. By a repeated application of Lemma~\ref{tbl:patternDecompose}, we arrive at
\begin{align*}
w^{(f)}=&\big(w_0(K_{2g-1})\big)\big(s_{b_{2g-3}}w_0(K_{2g-3})\big)\cdots\big(s_{b_{3}}w_0(K_{3})\big)\big(s_{b_{1}}w_0(K_{1})\big)w\\
&\big(w_0(K_{2})s_{b_{2}}\big)\big(w_0(K_{4})s_{b_{4}}\big)\cdots \big(w_0(K_{2g-2})s_{b_{2g-2}}\big),\\
w=&\big(w_0(K_1)s_{b_1}\big)\big(w_0(K_3)s_{b_3}\big)\cdots\big(w_0(K_{2g-3})s_{b_{2g-3}}\big)\big(w_0(K_{2g-1})\big)w^{(f)}\\
&\big(s_{b_{2g-2}}w_0(K_{2g-2})\big)\cdots\big(s_{b_{4}}w_0(K_{4})\big)\big(s_{b_{2}}w_0(K_{2})\big).
\end{align*}
Recall that if $j-i\geq2$, then $a_j-b_i\geq2$ so any $u$ in the parabolic subgroup generated by $K_j$ would commute with any $v$ in the parabolic subgroup generated by $K_i$.
Inside the above expression for $w$, $w^{(f)}$ commutes with all the factors on the right hand side so we can move it all the way to the right. We can also move all the $w_0(K_{2i-1})$'s all the way to the left and similarly move all the $w_0(K_{2i})$'s all the way to the right, leaving the $s_{b_i}$'s in the middle. Let $J=K_1\cup K_3\cup\cdots K_{2g-1}$, $J'=K_2\cup K_4\cdots K_{2g-2}$ so that $J\cap J'=\{b_1,b_2,\ldots,b_{f-1}\}$ is totally disconnected. We have that $w=w_0(J)w_0(J\cap J')w_0(J')w^{(f)}$. 

\noindent\textbf{Case 2:} $f=2g-1$ is odd and $K_f$ is of type l$_0$. Then
\begin{align*}
w=&\big(w_0(K_1)s_{b_1}\big)\big(w_0(K_3)s_{b_3}\big)\cdots\big(w_0(K_{2g-3})s_{b_{2g-3}}\big)w^{(f)}\big(w_0(K_{2g-1})\big)\\
&\big(s_{b_{2g-2}}w_0(K_{2g-2})\big)\cdots\big(s_{b_{4}}w_0(K_{4})\big)\big(s_{b_{2}}w_0(K_{2})\big).
\end{align*}
Now we can commute $w^{(f)}$ all the way to the left instead. Also let $J=K_1\cup K_3\cup\cdots K_{2g-1}$, $J'=K_2\cup K_4\cdots K_{2g-2}$ so that
\[
w=w^{(f)}w_0(J)w_0(J\cap J')w_0(J').
\]

\noindent\textbf{Case 3:} $f=2g$ is even and $K_f$ is of type r$_0$. Then
\begin{align*}
w=&\big(w_0(K_1)s_{b_1}\big)\big(w_0(K_3)s_{b_3}\big)\cdots\big(w_0(K_{2g-1})s_{b_{2g-1}}\big)\big(w_0(K_{2g})\big)w^{(f)}\\
&\big(s_{b_{2g-2}}w_0(K_{2g-2})\big)\cdots\big(s_{b_{4}}w_0(K_{4})\big)\big(s_{b_{2}}w_0(K_{2})\big).
\end{align*}
Let $J=K_1\cup K_3\cup\cdots K_{2g-1}$, $J'=K_2\cup K_4\cdots K_{2g}$. We have
\[
w=w_0(J)w_0(J\cap J')w_0(J')w^{(f)}.
\]

\noindent\textbf{Case 4:} $f=2g$ is even and $K_f$ is of type l$_0$. Then
\begin{align*}
w=&\big(w_0(K_1)s_{b_1}\big)\big(w_0(K_3)s_{b_3}\big)\cdots\big(w_0(K_{2g-1})s_{b_{2g-1}}\big)w^{(f)}\big(w_0(K_{2g})\big)\\
&\big(s_{b_{2g-2}}w_0(K_{2g-2})\big)\cdots\big(s_{b_{4}}w_0(K_{4})\big)\big(s_{b_{2}}w_0(K_{2})\big).
\end{align*}
Let $J=K_1\cup K_3\cup\cdots K_{2g-1}$, $J'=K_2\cup K_4\cdots K_{2g}$. We have 
\[
w=w^{(f)}w_0(J)w_0(J\cap J')w_0(J').
\]

The cases where $K_f$ is of type n can be done in the exact same way as either $K_f$ is of type r$_0$ or l$_0$. Continuing with the next connected components in $\{K_{f+1},\ldots,K_m\}$ and so on, we deduce that $w$ has the same form as in Definition~\ref{def:almost-parabolic} so it is polished.
\end{proof}
\begin{remark}
In this section, the purpose of distinguishing between type l and r is to specify the order of multiplying permutations in the decomposition of $w$. This order can also be seen as governed by the staircase diagram introduced by Richmond and Slofstra \cite{Richmond-Slofstra-staircase-diagrams}. We did not discuss the notion of staircase diagrams since they were not needed in full generality. 
\end{remark}

\subsection{Proof of direction (\ref{enum:w-is-polished})$\Rightarrow$(\ref{enum:interval-self-dual})}
\label{sec:proof-of-3-4}

We now prove the implication (\ref{enum:w-is-polished})$\Rightarrow$(\ref{enum:interval-self-dual}) for general finite Coxeter groups $W$.  Throughout this section $s_1 \ldots s_n$ is a generic reduced expression; we drop the convention from the previous section that $s_i$ is the specific simple reflection $(i \: i+1)$.

\begin{prop} \label{prop:different-support-gives-product}
Suppose that for $w \in W$ we can write $w=uv$ with $\supp(u) \cap \supp(v) = \emptyset$, then 
\[
[e,w] \cong [e,u] \times [e,v].
\]
\end{prop}
\begin{proof}
Let $J=\supp(v)$; since $D_R(u) \subseteq \supp(u) \subseteq S \setminus J$, we have $u \in W^J$, so in particular $\ell(w)=\ell(u)+\ell(v)$.  Let $u=s'_1 \cdots s'_m$ and $v=s_1 \cdots s_n$ be reduced expressions, then 
\[
w=s'_1 \cdots s'_m s_1 \cdots s_n
\]
is a reduced expression for $w$, with all $s'_i \in S \setminus J$ and all $s_j \in J$.  By Proposition \ref{prop:subword-property}, $[e,w]$ is the set of all reduced subwords of this word ordered by containment as subwords.  Any subword $\sigma$ of $s'_1 \cdots s'_m s_1 \cdots s_n$ consists of some elements of $S \setminus J$ followed by some elements of $J$, and by the above argument $\sigma$ is reduced if and only if each of these segments is reduced.  Thus multiplication gives an isomorphism of posets $[e,u] \times [e,v] \to [e,w]$.
\end{proof}

As products of self-dual posets are clearly self-dual, Proposition \ref{prop:different-support-gives-product} implies that it suffices to prove the implication (\ref{enum:w-is-polished})$\Rightarrow$(\ref{enum:interval-self-dual}) in the case where the polished element $w$ has a single block $S_1 = S$.  For the remainder of this section, let $w=w_0(J) \cap w_0(J \cap J') w_0(J')$ with $S=J \cup J'$ and $J \cap J'$ totally disconnected be such a polished element of $(W,S)$.

\begin{lemma} \label{lem:parabolic-decomp-of-almost-parabolic}
With $w=w_0(J)w_0(J \cap J')w_0(J')$ as above, we have 
\begin{align*}
    w_{J'}&=w_0(J'), \\
    w^{J'}&=w_0(J)w_0(J \cap J'),
\end{align*}
and this decomposition $w=w^{J'}w_{J'}$ is a BP-decomposition.
\end{lemma}
\begin{proof}
We know $w_0(J) \geq_L w_0(J \cap J')$ since $w_0(J)$ is the unique maximal element of $W_J$ under weak order, thus we may write 
\[
w_0(J)=s_1 \cdots s_k w_0(J \cap J')
\]
with lengths adding, for some reduced expression $s_1 \cdots s_k$ with each $s_i \in J$.  Since $w_0(J \cap J')$ is an involution, we see that 
\[
w_0(J)w_0(J \cap J')=s_1 \cdots s_k;
\]
furthermore, since $s_1 \cdots s_k w_0(J \cap J')$ was length-additive, we know that
\[
D_R(s_1 \cdots s_k) \cap (J \cap J') = \emptyset.
\]  
As $D_R(s_1 \cdots s_k) \subseteq J$, we conclude that $w_0(J)w_0(J \cap J')=s_1 \cdots s_k \in W^{J'}$.  Now,
\[
w=s_1 \cdots s_k w_0(J')
\]
is length-additive, so by uniqueness of parabolic decompositions we conclude $w_{J'}=w_0(J')$ and $w^{J'}=w_0(J)w_0(J \cap J')$.  Finally, it is trivially true that 
\[
(\supp(w^{J'}) \cap J') \subseteq J' = D_L(w_{J'}),
\]
so this is a BP-decomposition.
\end{proof}

Proposition 4.2 of \cite{Richmond-Slofstra-fibre-bundles} strengthens the following lemma, whose short proof we include for convenience:

\begin{lemma} \label{lem:bp-implies-interval-contains-product}
Let $u \in W$ and $K \subseteq S$ be such that $u=u^K u_K$ is a BP-decomposition, then the multiplication map
\[
[e,u^K]^K \times [e,u_K] \to [e,u]
\]
is an order-preserving bijection.
\end{lemma}
\begin{proof}
The map is injective by the uniqueness of parabolic decompositions.  To see surjectivity, suppose that $v \in [e,u]$, then by Proposition \ref{prop:parabolic-projection-preserves-bruhat} we have that $v^K \leq u^K$. On the other hand, by Proposition \ref{prop:bp-implies-parabolic-is-maximal}, we have $v_K \leq u_K$, since $v_K \leq v \leq u$ and $v_K \in W_K$.  Thus $v=v^Kv_K$ is in the image.  The order-preserving property is immediate from the fact that all products are length-additive and the subword characterization of Bruhat order in Proposition \ref{prop:subword-property}.
\end{proof}

\begin{remark} \label{rem:not-equal-to-product}
A word of caution when reading Lemma \ref{lem:bp-implies-interval-contains-product}: except in very special cases it is \emph{not} true that $[e,u^K]^K \times [e,u_K]$ and $[e,u]$ are isomorphic as posets, as $[e,u]$ may contain extra order relations not coming from the product.
\end{remark}

We are now ready to prove the implication (\ref{enum:w-is-polished})$\Rightarrow$(\ref{enum:interval-self-dual}) from Theorem \ref{thm:main}.

\begin{proof}[Proof of implication (\ref{enum:w-is-polished})$\Rightarrow$(\ref{enum:interval-self-dual}) from Theorem \ref{thm:main}]
Let $w$ be a polished element of $W$ with
\[
w=w_0(J)w_0(J \cap J')w_0(J'),
\]
we want to show that the interval $[e,w]$ is self-dual by exhibiting an explicit bijection $[e,w] \to [e,w]$ sending $u \mapsto u^{\vee}$ such that $u \leq v$ if and only if $v^{\vee} \leq u^{\vee}$ (an antiautomorphism).

We observe that 
\[
w^{J'}=w_0(J)w_0(J \cap J')=w_0(J)^{J\cap J'}.
\]
If $u \in [e,w_0(J)^{J \cap J'}]$, then $\supp(u) \subseteq J$, so $D_R(u) \subseteq J$.  Thus if $u \in W^{J \cap J'}$ we have in fact that $u \in W^{J'}$.  Thus we have that
\[
[e,w_0(J)w_0(J\cap J')]^{J'}=[e,w_0(J)^{J \cap J'}]^{J'}=[e,w_0(J)^{J \cap J'}]^{J \cap J'}=W_J^{J \cap J'}.
\]
Clearly we also have $[e,w_0(J')]=W_{J'}$ and so by Lemmas \ref{lem:parabolic-decomp-of-almost-parabolic} and \ref{lem:bp-implies-interval-contains-product} multiplication is an order preserving bijection
\[
W_J^{J \cap J'} \times W_{J'} \to [e,w].
\]

It is well known that $W_J^{J \cap J'}$ and $W_{J'}$ are self-dual as posets under Bruhat order with duality maps $u \mapsto w_0(J)uw_0(J \cap J')$ and $u \mapsto uw_0(J')$ respectively (see \cite{Bjorner-Brenti}).  This suggests the duality map
\[
u \mapsto u^{\vee} := w_0(J) u^{J'} w_0(J \cap J') \cdot u_{J'} w_0(J')
\]
for $[e,w]$.  Note that, by Remark \ref{rem:not-equal-to-product}, we still need to check whether this map is indeed an antiautomorphism of $[e,w]$ (indeed, up to this point we have not needed the assumption that $J \cap J'$ is totally disconnected).

Suppose we have a cover relation $u \lessdot v$ in $[e,w]$; to complete the proof we need to show that $v^{\vee} \lessdot u^{\vee}$.  Choose reduced decompositions of $v^{J'}$ and $v_{J'}$ to get a reduced decomposition 
\[
v=v^{J'}v_{J'}=(s_1 \cdots s_k) (s'_1 \cdots s'_{k'}).
\]
By Proposition \ref{prop:subword-property}, we know $u$ has a reduced decomposition obtained by omitting one of the simple generators above.  If the generator omitted is one of the $s'_i$, then we have $u^{J'}=v^{J'}$ and $u_{J'} \lessdot v_{J'}$ because $W_{J'}$ is an order ideal under Bruhat order.  In this case, the fact that our duality map is known to be an antiautomorphism for $W_J^{J \cap J'} \times W_{J'}$ implies that $v^{\vee} \lessdot u^{\vee}$.  

The case where the omitted generator is one of the $s_i$ needs another argument, as $W_J^{J \cap J'}$ is not an order ideal (so $u^{J'}$ may not equal $s_1 \cdots \hat{s_i} \cdots s_k$).  Suppose we are in this case, with $v^{J'}=s_1 \cdots s_k$, $v_{J'}=s'_1 \cdot s'_k$, and 
\[
u=s_1 \cdots \hat{s_i} \cdots s_k s'_1 \cdots s'_{k'},
\]
and all of these expressions reduced, and let $z=s_1 \cdots \hat{s_i} \cdots s_k$.  For convenience, we write $x$ for $_{J \cap J'}(v_{J'})$ and $y$ for $^{J \cap J'}(v_{J'})$ (so $xy=v_{J'}$ with lengths adding). Then we have length-additive products
\begin{align}
    v &= v^{J'} x y \label{eq:v-decomp} \\
    u &= z^{J'} z_{J'} x y \label{eq:u-decomp}.
\end{align}
Since $z_{J'}, x,$ and $y$ are all in $W_{J'}$, so is their product.  And since the above decomposition $u=z^{J'} (z_{J'}xy)$ is length-additive, uniqueness of parabolic decompositions implies that $z^{J'}=u^{J'}$ and $z_{J'}xy=u_{J'}$.  Also, because $y \in ^{J \cap J'}W_{J'}$ has no left descents from $J \cap J'$, we know that $yw_0(J')$ has all elements of $J \cap J'$ as descents, and therefore $y \geq_R w_0(J \cap J')$, so we may write $yw_0(J)=w_0(J \cap J')y'$ for some element $y'$ with $\ell(y)=\ell(w_0(J \cap J'))+\ell(y')$.

Now, we have
\begin{align*}
    u^{\vee}&=w_0(J) u^{J'} w_0(J \cap J') u_{J'} w_0(J') \\
    &=w_0(J) u^{J'} w_0(J \cap J') z_{J'} x y w_0(J') \\
    &=w_0(J) u^{J'} w_0(J \cap J') z_{J'} x w_0(J \cap J') y' \\
    &=w_0(J) u^{J'} z_{J'} x y' \\
\end{align*}
where in the last step we have used that $z_{J'}x \in W_{J \cap J'}$, which is abelian by our assumption that $J \cap J'$ is totally disconnected, and therefore commutes with $w_0(J \cap J')$.  Similarly, we have
\[
v^{\vee}=w_0(J)v^{J'}xy'.
\]

In the following computation, we write $N_K$ for $\ell(w_0(K))$ for any subset $K \subseteq S$.  Computing lengths, we have 
\begin{align*}
    \ell(u^{\vee})&= \ell(w)-\ell(u) \\
    &=(N_J+N_{J'}-N_{J \cap J'})-(\ell(u^{J'})+\ell(z_{J'})+\ell(x)+\ell(y)) \\
    &= (N_J-\ell(u^{J'})-\ell(z_{J'})-\ell(x))+\ell(y')
\end{align*}
where in the first step we have used the length-additive decomposition (\ref{eq:u-decomp}) and in the second we have used the fact that $yw_0(J')=w_0(J \cap J')y'$ with the right-hand-side being length-additive, and the left-hand-side having length $N_{J'}-\ell(y)$.  This implies that 
\[
u^{\vee}=(w_0(J)u^{J'}z_{J'}x) \cdot y'
\]
is length-additive.  A similar calculation shows that 
\[
v^{\vee}=(w_0(J)v^{J'}x) \cdot y'
\]
is also length-additive.  Thus $v^{\vee} \lessdot u^{\vee}$ if and only if 
\[
w_0(J)v^{J'}x \lessdot w_0(J)u^{J'}z_{J'}x,
\]
which, because $w_0(J)$ is an antiautomorphism of Bruhat order on $W_J$, is true in turn if and only if $u^{J'}z_{J'}x \lessdot v^{J'}x$.  These decompositions are length-additive, as they come from parabolic decompositions, thus we need to check that $u^{J'}z_{J'} \lessdot v^{J'}$.  Finally we see this is true by recalling that 
\[
u^{J'}z_{J'}=z^{J'}z_{J'}=z=s_1 \cdots \hat{s_i} \cdots s_k
\]
and $v^{J'}=s_1 \cdots s_k$.  This completes the proof of implication (\ref{enum:w-is-polished})$\Rightarrow$(\ref{enum:interval-self-dual}).
\end{proof}

\subsection{Proof of Theorem~\ref{thm:our-top-heavy}}
We obtain Theorem~\ref{thm:our-top-heavy} as a corollary of the already established Theorem~\ref{thm:main}, with technology similar to that of Section~\ref{sub:1-2}.
\begin{proof}[Proof of Theorem~\ref{thm:our-top-heavy}]
Let $w$ be smooth so that it avoids 3412 and 4231. We will show that if $w$ contains one of the patterns 34521, 45321, 54123 and 54312, then 
\[
\max_{u\in P_1^w}\udeg_w(u)<\max_{u\in P_{\ell(w)-1}^w}\ddeg_w(u).
\]
On the other hand, we know from Theorem~\ref{thm:main} that if $w$ avoids these patterns, then $[e,w]$ in the Bruhat order is self-dual and clearly
\[
\max_{u\in P_1^w}\udeg_w(u)=\max_{u\in P_{\ell(w)-1}^w}\ddeg_w(u).
\]
Thus, throughout the rest of the proof, assume that $w$ contains one of 34521, 45321, 54123 or 54312. 

We use induction on $n$ to show that for any $u\in P_1^w$, $\udeg_w(u)-|P_1^w|\leq1$, and that there exists some $u\in P_{\ell(w)-1}^w$ such that $\ddeg_w(u)-|P_1^w|\geq2$. This statement suffices for the sake of the theorem.

We first reduce to the case where $w$ does not lie in any proper parabolic subgroup of $\mathfrak{S}_n$, or in other words, $\bl(w)=1$, with the notation defined in Section~\ref{sub:1-2}. Let $b=\bl(w)\geq2$ and $w=w^{(1)}\oplus\cdots\oplus w^{(b)}$. Now the Bruhat interval can be factored as 
\[
[e,w] \cong [e,w^{(1)}]\times\cdots \times[e,w^{(b)}].
\]
Each factor $w^{(i)}$ avoids 3412 and 4231 and is thus smooth, so that $[e,w^{(i)}]$ is rank symmetric. Take $u\in[e,w]$ and write it as $u^{(1)}\oplus\cdots\oplus u^{(b)}$ corresponding to the decomposition of $w$. If $\ell(u)=1$, there exists some $j\in\{1,\ldots,b\}$ such that $u^{(i)}=e$ for all $i\neq j$. Then
$$\udeg_w(u)=\sum_{i\neq j}|P_1^{w(i)}|+\udeg_{w^{(j)}}(u^{(j)})=|P_1^w|+\udeg_{w^{(j)}}(u^{(j)})-|P_1^{w(j)}|.$$
By the induction hypothesis, $\udeg_{w^{(j)}}(u^{(j)})-|P_1^{w(j)}|\leq1$ so $\udeg_w(u)-|P_1^w|\leq1$. On the other hand, since all the four patterns of interest do not lie in any proper parabolic subgroup of $\mathfrak{S}_4$, there exists some $w^{(j)}$ containing one of the patterns. By induction hypothesis, there exists some $u^{(j)}\in P_{\ell(w^{(j)})-1}^{w^{(j)}}$ such that $\ddeg_{w^{(j)}}(u^{(j)})-|P_1^{w^{(j)}}|\geq2$. Construct $u=u^{(1)}\oplus\cdots\oplus u^{(b)}\in P_{\ell(w)-1}^w$ where $u^{(i)}=w^{(i)}$ for $i\neq j$. Similarly, we see that
\begin{align*}
\ddeg_w(u)=&\sum_{i\neq j}|P_{\ell(w^{(i)})-1}^{w^{(i)}}|+\ddeg_{w^{(j)}}(u^{(j)})\\
\geq&\sum_{i\neq j}|P_{1}^{w^{(i)}}|+|P_1^{w^{(j)}}|+2\\
=&|P_1^w|+2.
\end{align*}

Now we know that $w$ does not lie in any proper parabolic subgroup of $\mathfrak{S}_n$. This means $P_1^w=\{s_1,\ldots,s_{n-1}\}$ contains all simple transpositions. For any $s_i$, the permutations that cover $s_i$ in $P_2^w$ are contained in $$\{s_1s_i,s_2s_i,\ldots,s_{i-1}s_i,s_{i+1}s_i,\ldots,s_{n-1}s_i\}\cup\{s_is_{i-1},s_is_{i+1}\}$$
which has cardinality $n$ if $i\in\{2,\ldots,n-2\}$ and cardinality $n-1$ if $i=1,n-1$. As a result, $\udeg_w(u)\leq n$ for all $u\in P_1^w$. In other words, $\udeg_w(u)-|P_1^w|\leq1$.

Next, we obtain a lower bound of $n+1$ for $\ddeg_w(u)$ for some $u\in P^w_{\ell(w)-1}$. Recall the notion of a \textit{minimal inversion} from Definition~\ref{def:minimal-inversion}. The number of minimal inversions of $w$ is exactly $|P_{\ell(w)-1}^w|=|P_1^w|=n-1$. Suppose that $(i_1,j_1)$ and $(i_2,j_2)$ are two minimal inversions of $w$ with $i_1\leq i_2$, we claim that there exists some $v\in[e,w]$ covered by both $wt_{i_1j_1}$ and $wt_{i_2j_2}$ in the Bruhat order. Consider the following cases. If $\{i_1,j_1\}$ and $\{i_2,j_2\}$ are not disjoint, then either $j_1=j_2$ or $i_1=i_2$ or $i_2=j_1$. If $j_1=j_2$, then $w(i_1)<w(i_2)$ by minimality, and $v=wt_{i_1j_1}t_{i_2j_2}=wt_{i_2j_2}t_{i_1i_2}$ is covered by both. The case $i_1=i_2$ is the same. And if $i_2=j_1$, then by Lemma~\ref{lem:adj-labels}, there are two such $v$'s that serve the purpose. If $\{i_1,j_1\}$ and $\{i_2,j_2\}$ are disjoint, then $t_{i_1j_1}$ and $t_{i_2j_2}$ commute. Pictorially, we just need to check that in the permutation diagram, the rectangle formed by $(i_1,w(i_1))$ and $(j_1,w(j_1))$ is disjoint from the rectangle formed by $(i_2,w(i_2))$ and $(j_2,w(j_2))$ so that $v=wt_{i_1j_1}t_{i_2j_2}$ is covered by both $wt_{i_1j_1}$ and $wt_{i_2j_2}$. These two rectangles overlap precisely when $i_1<i_2<j_1<j_2$ and $w(i_2)<w(i_1)>w(j_2)>w(j_1)$. However, in this case, $w$ contains 3412 at indices $i_1,i_2,j_1,j_2$, contradicting $w$ being smooth.

Fix a minimal inversion $(p,q)$ of $w$. For other $n-2$ minimal inversions $(i,j)$, let $V_{(i,j)}=\{v\in P^w_{\ell(w)-2}\:|\: v<wt_{pq},v<wt_{ij}\}$. Since every Bruhat interval of rank 2 is isomorphic to a diamond (see for example \cite{Bjorner-Brenti}), we know that every $v\in P^w_{\ell(w)-2}$ such that $v<wt_{pq}$ belongs to exactly one of $V_{(i,j)}$'s. This means $\ddeg_w(wt_{pq})$ is the sum of $|V_{(i,j)}|$'s. Moreover, we have seen that $|V_{(i,j)}|\geq1$ for all minimal inversions $(i,j)\neq(p,q)$ from the previous paragraph and that $|V_{(i,j)}|\geq2$ if $i=q$ or $j=p$ from Lemma~\ref{lem:adj-labels}. As a result, if there are at least three minimal inversions $(i,j)$ of $w$ such that $i=q$ or $j=p$, we know that $\ddeg_w(wt_{pq})\geq n+1$.

We apply arguments as in the proof of Lemma~\ref{lem:length5patterns}. If $w$ contains 45321, take a minimal pattern 45321 in the sense of Definition~\ref{def:pattern-contain} at indices $a_1<a_2<a_3<a_4<a_5$ as in Figure~\ref{fig:45321} where all the regions $A_{*,*}$'s are empty. Let $(p,q)=(a_3,a_4)$. Since $(a_1,a_3)$, $(a_2,a_3)$ and $(a_4,a_5)$ are all minimal inversions, we know that $\ddeg_w(wt_{pq})\geq n+1$. The case of 54312, which is the inverse of 45321, is the same. If $w$ avoids 45321 and 54312 but contains 34521, we take a minimal pattern as in Figure~\ref{fig:34521}. With notations in the proof of Lemma~\ref{lem:length5patterns}, we let $(p,q)=(a_4,a_5)$. Since $(a_3,a_4)$, $(a_1,a_4)$ and $(c_k,a_4)$ are all minimal inversions, we also conclude that $\ddeg_w(wt_{pq})\geq n+1$. The case of 54123, which is the inverse of 34521, is the same. In both cases, $\ddeg_w(wt_{pq})\geq n+1$ so we are done.
\end{proof}

\section{Discussion of other types}
\label{sec:counterexamples}

Theorem \ref{thm:main} fails in general finite Coxeter groups, in particular we have the following counterexamples for (\ref{enum:graphs-isomorphic})$\Rightarrow$(\ref{enum:interval-self-dual}) and (\ref{enum:interval-self-dual})$\Rightarrow$(\ref{enum:w-is-polished}):

\begin{itemize}
    \item For $(W,S)$ if type $B_3$ with generators chosen so that $(s_1s_2)^3=e$ and $(s_2s_3)^4=e$, the element 
    \[
    w=s_3s_2s_3s_1s_2s_3s_1s_2
    \]
    has $\Gamma_w \cong \Gamma^w$, but $[e,w]$ is not self-dual.
    \item The two elements of length three in $W$ of type $B_2$ have $[e,w]$ self-dual, but are not polished.
\end{itemize}

There is a notion of pattern avoidance for general finite Weyl groups (see \cite{Billey-Postnikov}).  This notion was introduced by Billey and Postnikov in order to give a generalization of the Lakshmibai-Sandhya smoothness criterion for Schubert varieties.  We do not know whether self-dual Bruhat intervals in types other than $A_{n-1}$ are characterized by pattern avoidance.

\begin{quest}
Is the set of elements $w$ of finite Weyl groups such that $[e,w]$ is self-dual characterized by pattern avoidance in the sense of \cite{Billey-Postnikov} as in (\ref{enum:w-is-polished})?
\end{quest}

\section*{Acknowledgements}

We are grateful to Sara Billey for suggesting that self-dual intervals may be characterized by pattern avoidance.  We also wish to thank Alexander Woo for providing helpful references and Alexander Postnikov and Thomas Lam for their suggestions.

\bibliographystyle{plain}
\bibliography{strongdual}
\end{document}